\documentclass[a4paper]{amsart}

\usepackage[T1]{fontenc}
\usepackage{amstext}
\usepackage{amsthm}
\usepackage{amssymb}
\usepackage{stackrel}
\usepackage[unicode=true,pdfusetitle,
 bookmarks=true,bookmarksnumbered=false,bookmarksopen=false,
 breaklinks=false,pdfborder={0 0 0},pdfborderstyle={},backref=false,colorlinks=false]
 {hyperref}

\makeatletter

\theoremstyle{plain}
\newtheorem{thm}{\protect\theoremname}[section]
\theoremstyle{definition}
\newtheorem{defn}[thm]{\protect\definitionname}
\theoremstyle{remark}
\newtheorem{rem}[thm]{\protect\remarkname}
\theoremstyle{plain}
\newtheorem{cor}[thm]{\protect\corollaryname}
\theoremstyle{plain}
\newtheorem{prop}[thm]{\protect\propositionname}
\theoremstyle{definition}
\newtheorem{notation}[thm]{\protect\notationname}

\usepackage{tikz-cd}
\usepackage{mathtools}
\usepackage{adjustbox}
\usepackage{enumitem}
\tikzcdset{scale cd/.style={every label/.append style={scale=#1},
    cells={nodes={scale=#1}}}}
\usepackage{eucal}
\usepackage{mleftright}
\mleftright
\DeclareMathSymbol{:}{\mathpunct}{operators}{"3A}
\usepackage{xurl}

\makeatother

\providecommand{\corollaryname}{Corollary}
\providecommand{\definitionname}{Definition}
\providecommand{\notationname}{Notation}
\providecommand{\propositionname}{Proposition}
\providecommand{\remarkname}{Remark}
\providecommand{\theoremname}{Theorem}

\begin{document}
\global\long\def\sf#1{\mathsf{#1}}%

\global\long\def\scr#1{\mathscr{{#1}}}%

\global\long\def\cal#1{\mathcal{#1}}%

\global\long\def\bb#1{\mathbb{#1}}%

\global\long\def\frak#1{\mathfrak{#1}}%

\global\long\def\fr#1{\mathfrak{#1}}%

\global\long\def\u#1{\underline{#1}}%

\global\long\def\tild#1{\widetilde{#1}}%

\global\long\def\mrm#1{\mathrm{#1}}%

\global\long\def\pr#1{\left(#1\right)}%

\global\long\def\abs#1{\left|#1\right|}%

\global\long\def\inp#1{\left\langle #1\right\rangle }%

\global\long\def\br#1{\left\{  #1\right\}  }%

\global\long\def\norm#1{\left\Vert #1\right\Vert }%

\global\long\def\hat#1{\widehat{#1}}%

\global\long\def\opn#1{\operatorname{#1}}%

\global\long\def\bigmid{\,\middle|\,}%

\global\long\def\Top{\sf{Top}}%

\global\long\def\Set{\sf{Set}}%

\global\long\def\SS{\sf{sSet}}%

\global\long\def\BS{\sf{bsSet}}%

\global\long\def\Kan{\sf{Kan}}%

\global\long\def\Cat{\sf{Cat}}%

\global\long\def\RC{\sf{RelCat}}%

\global\long\def\Grpd{\mathcal{G}\sf{rpd}}%

\global\long\def\Res{\mathcal{R}\sf{es}}%

\global\long\def\imfld{\cal M\mathsf{fld}}%

\global\long\def\ids{\cal D\sf{isk}}%

\global\long\def\ich{\cal C\sf h}%

\global\long\def\SW{\mathcal{SW}}%

\global\long\def\SHC{\mathcal{SHC}}%

\global\long\def\Fib{\mathcal{F}\mathsf{ib}}%

\global\long\def\Bund{\mathcal{B}\mathsf{und}}%

\global\long\def\Fam{\cal F\sf{amOp}}%

\global\long\def\B{\sf B}%

\global\long\def\Spaces{\sf{Spaces}}%

\global\long\def\Mod{\sf{Mod}}%

\global\long\def\Nec{\sf{Nec}}%

\global\long\def\Fin{\sf{Fin}}%

\global\long\def\Ch{\sf{Ch}}%

\global\long\def\Ab{\sf{Ab}}%

\global\long\def\SA{\sf{sAb}}%

\global\long\def\P{\mathsf{POp}}%

\global\long\def\Op{\mathcal{O}\mathsf{p}}%

\global\long\def\Opg{\mathcal{O}\mathsf{p}_{\infty}^{\mathrm{gn}}}%

\global\long\def\Tup{\mathsf{Tup}}%

\global\long\def\H{\cal H}%

\global\long\def\Mfld{\cal M\mathsf{fld}}%

\global\long\def\Disk{\cal D\mathsf{isk}}%

\global\long\def\Acc{\mathcal{A}\mathsf{cc}}%

\global\long\def\Pr{\mathcal{P}\mathrm{\mathsf{r}}}%

\global\long\def\Del{\mathbf{\Delta}}%

\global\long\def\id{\operatorname{id}}%

\global\long\def\Aut{\operatorname{Aut}}%

\global\long\def\End{\operatorname{End}}%

\global\long\def\Hom{\operatorname{Hom}}%

\global\long\def\Ext{\operatorname{Ext}}%

\global\long\def\sk{\operatorname{sk}}%

\global\long\def\ihom{\underline{\operatorname{Hom}}}%

\global\long\def\N{\mathrm{N}}%

\global\long\def\-{\text{-}}%

\global\long\def\op{\mathrm{op}}%

\global\long\def\To{\Rightarrow}%

\global\long\def\rr{\rightrightarrows}%

\global\long\def\rl{\rightleftarrows}%

\global\long\def\mono{\rightarrowtail}%

\global\long\def\epi{\twoheadrightarrow}%

\global\long\def\comma{\downarrow}%

\global\long\def\ot{\leftarrow}%

\global\long\def\corr{\leftrightsquigarrow}%

\global\long\def\lim{\operatorname{lim}}%

\global\long\def\colim{\operatorname{colim}}%

\global\long\def\holim{\operatorname{holim}}%

\global\long\def\hocolim{\operatorname{hocolim}}%

\global\long\def\Ran{\operatorname{Ran}}%

\global\long\def\Lan{\operatorname{Lan}}%

\global\long\def\Sk{\operatorname{Sk}}%

\global\long\def\Sd{\operatorname{Sd}}%

\global\long\def\Ex{\operatorname{Ex}}%

\global\long\def\Cosk{\operatorname{Cosk}}%

\global\long\def\Sing{\operatorname{Sing}}%

\global\long\def\Sp{\operatorname{Sp}}%

\global\long\def\Spc{\operatorname{Spc}}%

\global\long\def\Fun{\operatorname{Fun}}%

\global\long\def\map{\operatorname{map}}%

\global\long\def\diag{\operatorname{diag}}%

\global\long\def\Gap{\operatorname{Gap}}%

\global\long\def\cc{\operatorname{cc}}%

\global\long\def\ob{\operatorname{ob}}%

\global\long\def\Map{\operatorname{Map}}%

\global\long\def\Rfib{\operatorname{RFib}}%

\global\long\def\Lfib{\operatorname{LFib}}%

\global\long\def\Tw{\operatorname{Tw}}%

\global\long\def\Equiv{\operatorname{Equiv}}%

\global\long\def\Arr{\operatorname{Arr}}%

\global\long\def\Cyl{\operatorname{Cyl}}%

\global\long\def\Path{\operatorname{Path}}%

\global\long\def\Alg{\operatorname{Alg}}%

\global\long\def\ho{\operatorname{ho}}%

\global\long\def\Comm{\operatorname{Comm}}%

\global\long\def\Triv{\operatorname{Triv}}%

\global\long\def\triv{\operatorname{triv}}%

\global\long\def\Env{\operatorname{Env}}%

\global\long\def\Act{\operatorname{Act}}%

\global\long\def\loc{\operatorname{loc}}%

\global\long\def\Assem{\operatorname{Assem}}%

\global\long\def\Nat{\operatorname{Nat}}%

\global\long\def\Conf{\operatorname{Conf}}%

\global\long\def\Rect{\operatorname{Rect}}%

\global\long\def\Emb{\operatorname{Emb}}%

\global\long\def\Homeo{\operatorname{Homeo}}%

\global\long\def\mor{\operatorname{mor}}%

\global\long\def\Germ{\operatorname{Germ}}%

\global\long\def\Post{\operatorname{Post}}%

\global\long\def\Sub{\operatorname{Sub}}%

\global\long\def\Shv{\operatorname{Shv}}%

\global\long\def\Cls{\operatorname{Cls}}%

\global\long\def\Seg{{\rm Seg}}%

\global\long\def\hc{\mathrm{hc}}%

\global\long\def\bi{\mathrm{bi}}%

\global\long\def\CSS{\mathrm{CSS}}%

\global\long\def\disj{\mathrm{disj}}%

\global\long\def\eq{\mathrm{eq}}%

\global\long\def\Top{\mathrm{Top}}%

\global\long\def\lax{\mathrm{lax}}%

\global\long\def\weq{\mathrm{weq}}%

\global\long\def\fib{\mathrm{fib}}%

\global\long\def\inert{\mathrm{inert}}%

\global\long\def\act{\mathrm{act}}%

\global\long\def\cof{\mathrm{cof}}%

\global\long\def\inj{\mathrm{inj}}%

\global\long\def\univ{\mathrm{univ}}%

\global\long\def\Ker{\opn{Ker}}%

\global\long\def\Coker{\opn{Coker}}%

\global\long\def\Im{\opn{Im}}%

\global\long\def\Coim{\opn{Im}}%

\global\long\def\coker{\opn{coker}}%

\global\long\def\im{\opn{\mathrm{im}}}%

\global\long\def\coim{\opn{coim}}%

\global\long\def\gn{\mathrm{gn}}%

\global\long\def\Mon{\opn{Mon}}%

\global\long\def\Un{\opn{Un}}%

\global\long\def\St{\opn{St}}%

\global\long\def\cun{\widetilde{\opn{Un}}}%

\global\long\def\cst{\widetilde{\opn{St}}}%

\global\long\def\Sym{\operatorname{Sym}}%

\global\long\def\CA{\operatorname{CAlg}}%

\global\long\def\Ind{\operatorname{Ind}}%

\global\long\def\rd{\mathrm{rd}}%

\global\long\def\xmono#1#2{\stackrel[#2]{#1}{\rightarrowtail}}%

\global\long\def\xepi#1#2{\stackrel[#2]{#1}{\twoheadrightarrow}}%

\global\long\def\adj{\stackrel[\longleftarrow]{\longrightarrow}{\bot}}%

\global\long\def\btimes{\boxtimes}%

\global\long\def\ps#1#2{\prescript{}{#1}{#2}}%

\global\long\def\ups#1#2{\prescript{#1}{}{#2}}%

\global\long\def\hofib{\mathrm{hofib}}%

\global\long\def\cofib{\mathrm{cofib}}%

\global\long\def\Vee{\bigvee}%

\global\long\def\w{\wedge}%

\global\long\def\t{\otimes}%

\global\long\def\bp{\boxplus}%

\global\long\def\rcone{\triangleright}%

\global\long\def\lcone{\triangleleft}%

\global\long\def\S{\mathsection}%

\global\long\def\p{\prime}%

\global\long\def\pp{\prime\prime}%

\global\long\def\W{\overline{W}}%

\global\long\def\o#1{\overline{#1}}%

\global\long\def\fp{\overrightarrow{\times}}%

\title{Classification diagrams of simplicial categories}
\author{Kensuke Arakawa}
\date{}
\begin{abstract}
We show that the classification diagram of a relative $\infty$-category
arising from a relative simplicial category is equivalent to the levelwise
nerve. Applications include the comparison of the diagonal of the
levelwise nerve and the homotopy coherent nerve, and a result on the
levelwise localizations of simplicial categories.
\end{abstract}
\address{Department of Mathematics, Kyoto  University,Kitashirakawa-Oiwakecho, 6068502 Kyoto (Japan)},
\email{arakawa.kensuke.22c@st.kyoto-u.ac.jp}
\maketitle



\section{Introduction}

Given a category $\cal C$, the \textbf{nerve} of $\cal C$ is the
simplicial set $N\pr{\cal C}$ whose $n$-simplices are the set of
$n$ composable arrows $X_{0}\to\cdots\to X_{n}$ in $\cal C$. It
is immediate from the definitions that the association $\cal C\mapsto N\pr{\cal C}$
defines a fully faithful functor
\[
N:\Cat\to\SS
\]
from the category of small categories to that of simplicial sets.
As elementary as this observation may be, it puts forward an amazing
fact: A structured object (a $1$-category) can be presented as a
family of less structured objects (sets, or $0$-categories). Experience
suggests that the latter is generally easier to work with, and this
explains the prevalence of the nerve construction in homotopy theory.

It is natural to look for a higher-categorical analog of the nerve construction.
For example, we may ask whether it is possible to present an $\pr{\infty,1}$-category
as a simplicial $\pr{\infty,0}$-category, or a simplicial space.
The answer is yes. Recall that \textbf{$\infty$-categories} (alias
quasicateogries \cite{Joyal_qcat_Kan}) are a certain class of simplicial
sets modeling $\pr{\infty,1}$-categories. Given an $\infty$-category
$\cal C$, its \textbf{classifying diagram} $\Cls\pr{\cal C}$ is
the simplicial space (or a bisimplicial set, to be more precise) whose
$n$th space $\Cls\pr{\cal C}_{n}=\Cls\pr{\cal C}_{n,\ast}$ is the
maximal sub Kan complex of $\Fun\pr{\Delta^{n},\cal C}$. Joyal and
Tierney showed in \cite{JT07} that the association $\cal C\to\Cls\pr{\cal C}$
defines a fully faithful functor from the homotopy category of $\infty$-categories
to that of simplicial spaces, thereby giving an analog of the nerve
construction; the essential image of this functor consists of Rezk's
\textbf{complete Segal spaces} \cite{Rezk01}.

A generalization of the classifying diagram construction has proved
to be especially important in the theory of localizations of $\infty$-categories
(in the sense of \cite[Definition 2.4.2]{Landoo-cat}). Given an $\infty$-category
$\cal C$ and a subcategory $\cal W\subset\cal C$ containing all
objects, their \textbf{classification diagram} $\Cls\pr{\cal C,\cal W}$
\cite{Rezk01} is the simplicial $\infty$-category (hence a bisimplicial
set) whose $n$th $\infty$-category is given by
\[
\Cls\pr{\cal C,\cal W}_{n}=\Cls\pr{\cal C,\cal W}_{n,\ast}=\Fun\pr{\Delta^{n},\cal C}\times_{\cal C^{n+1}}\cal W^{n+1}.
\]
For example, if $\cal W$ is the subcategory of equivalences of $\cal C$,
then $\Cls\pr{\cal C,\cal W}$ is nothing but the classifying diagram
of $\cal C$. Inspired by earlier works such as \cite{Rezk01, MR2439415, MR2877402},
Mazel-Gee showed\footnote{To be more precise, Mazel-Gee proved this in the case where $\cal W$
contains all equivalences (which suffices for most applications).
The general case was proved by the author in \cite{A23b}.} in \cite{MR4045352} (see also \cite{A23b, AC25} for alternative arguments
and generalizations) that a functor $f:\cal C\to\cal D$ of $\infty$-categories
which carries $\cal W$ into $\cal D^{\simeq}$ is a localization
with respect to $\cal W$ if and only if the map $\Cls\pr{\cal C,\cal W}\to\Cls\pr{\cal D,\cal D^{\simeq}}=\Cls\pr{\cal D}$
is a weak equivalence of the complete Segal space model structure
\cite[Theorem 7.2]{Rezk01}. In other words, up to fibrant replacement,
$\Cls\pr{\cal C,\cal W}$ computes the localization of $\cal C$.
This is useful because in some cases, the fibrant replacement of $\Cls\pr{\cal C,\cal W}$
have explicit descriptions; see, e.g., \cite{LM15}.

Now recall the \textbf{homotopy coherent nerve functor}
\[
N_{\hc}:\Cat_{\Delta}\to\SS,
\]
due to Cordier \cite{Cordier_hcnerve}, is a right Quillen equivalence
from Bergner's model structure on the category of simplicial categories
\cite{BergnerSimpCat} to Joyal's model structure on the category
of simplicial sets \cite[$\S$2.2.5]{HTT}. Many important $\infty$-categories
arise as the homotopy coherent nerves of simplicial categories, so
we frequently want to consider bisimplicial sets of the form $\Cls\pr{N_{\hc}\pr{\cal C},N_{\hc}\pr{\cal W}}$,
where $\cal C$ is a fibrant simplicial category and $\cal W$ is
its simplicial subcategory such that $N_{\hc}\pr{\cal W}$ is a subcategory
of $N_{\hc}\pr{\cal C}$. However, as is anything constructed from
homotopy coherent nerves, this bisimplicial set is somewhat difficult
to manipulate by hands. It will therefore be nice if there is an alternative,
preferably simpler, presentation of this bisimplicial set. The goal
of this paper is to show that this is possible, at least if we consider
the \textit{marked} version of classification diagrams.

A \textbf{marked bisimplicial set} is a pair $\pr{X,S}$ of a bisimplicial
set $X$ and a simplicial subset $S\subset X_{1}=X_{1,\ast}$ which
contains the image of the map $X_{0}\to X_{1}$. Equivalently, it
is a simplicial object $\{\pr{X_{\ast,n},S_{n}}\}_{n\geq0}$ in the
category of marked simplicial sets. One may argue that the natural
place where classification diagrams live is not the category of bisimplicial
sets, but the category of\textit{ }marked bisimplicial sets. Indeed,
there is a very natural functor
\[
\Cls^{+}:\SS^{+}\to\pr{\SS^{+}}^{\Del^{\op}}=\BS^{+}
\]
from the category of marked simplicial sets to the category of marked
bisimplicial sets, defined by $\pr{X,S}\mapsto\pr{X,S}^{\pr{\Delta^{\bullet}}}$.
(Here we wrote $\pr{X,S}^{\Delta^{n}}=\pr{X,S}^{\pr{\Delta^{n}}^{\sharp}}$,
where $\pr{\Delta^{n}}^{\sharp}$ denotes the standard simplex $\Delta^{n}$
with all edges marked. It is the cotensor of $\pr{X,S}$ by the simplicial
set $\Delta^{n}$ with respect to the simplicial enrichment $\Map^{\sharp}\pr{-,-}$
of \cite[$\S$3.1.3]{HTT}.) Unwinding the definitions, we find that
$\Cls\pr{\cal C,\cal W}$ is nothing but the underlying bisimplicial
set of $\Cls^{+}\pr{\cal C,\cal W_{1}}$. 

In \cite{A23b}, the author showed that marked bisimplicial sets are
to complete Segal spaces what marked simplicial sets are to $\infty$-categories.
More precisely, \cite[Theorems 2.9 and  3.4]{A23b} state that $\BS^{+}$
admits a model structure, denoted by $\BS_{\CSS}^{+}$, such that:
\begin{itemize}
\item The functor $\Cls^{+}:\SS^{+}\to\BS^{+}$ is a right Quillen equivalence
, where $\SS^{+}$ carries the model structure for marked simplicial
sets \cite[Proposition 3.1.3.7, Remark 3.1.4.6]{HA}.\footnote{In \cite{A23b}, the functor $\Cls$ is denoted by $N$ and the functor
$\Cls^{+}$ is denoted by $\pr{t^{+}}^{!}$. We changed the notation
in the hope that the paper will be more readable, for we will consider
many variations of nerves in this paper.}
\item The forgetful functor $\BS_{\CSS}^{+}\to\BS_{\CSS}$ is also a right
Quillen equivalence, where $\BS_{\CSS}$ denotes the category of bisimplicial
sets equipped with the model structure for complete Segal spaces.
\end{itemize}
What we will do is to construct a relatively simple marked bisimplicial
set which is weakly equivalent to $\Cls^{+}\pr{N_{\hc}\pr{\cal C},N_{\hc}\pr{\cal W}}$
in $\BS_{\CSS}^{+}$. To give a precise statement of the main theorem,
we must introduce some notation and terminology.
\begin{defn}
\label{def:relscat}Let $\cal C$ be a simplicial category. A simplicial
subcategory $\cal W\subset\cal C$ is said to be \textbf{wide} if
it satisfies the following pair of conditions:
\begin{itemize}
\item $\cal W$ contains all objects of $\cal C$.
\item For every pair of objects $X,Y\in\cal C$, the simplicial subset $\cal W\pr{X,Y}\subset\cal C\pr{X,Y}$
is a union of components of $\cal C\pr{X,Y}$. 
\end{itemize}
The pair $\pr{\cal C,\cal W}$ of a simplicial category and its wide
simplicial subcategory will be called a \textbf{relative simplicial
category}.
\end{defn}

\begin{defn}
Let $\cal C$ be a simplicial category. For each $n\geq0$, we let
$\cal C_{n}$ denote the ordinary category constructed from the objects
of $\cal C$ and the $n$-simplices of the hom-simplicial sets of
$\cal C$. The \textbf{binerve} (or the \textbf{levelwise nerve})
of $\cal C$ is the bisimplicial set $N_{\bi}\pr{\cal C}$ whose $n$th
row $N_{\bi}\pr{\cal C}_{\ast,n}$ is the nerve of $\cal C_{n}$.
Note that for each $n\geq1$, the $n$th column $N_{\bi}\pr{\cal C}_{n,\ast}$
of $N_{\bi}\pr{\cal C}$ is the disjoint union
\[
\coprod_{X_{0},\dots,X_{n}\in\cal C}\cal C\pr{X_{n-1},X_{n}}\times\cdots\times\cal C\pr{X_{0},X_{1}}.
\]
If $\cal W\subset\cal C$ is a wide simplicial subcategory, we let
$N_{\bi}^{+}\pr{\cal C,\cal W}$ denote the marked bisimplicial set
$\pr{N_{\bi}\pr{\cal C},\coprod_{X,Y\in\cal C}\cal W\pr{X,Y}}$. 
\end{defn}

We can now state the main result of this paper.
\begin{thm}
[Theorem \ref{thm:main}]\label{thm:intro}Let $\cal C$ be a fibrant
simplicial category and let $\cal W\subset\cal C$ be a wide simplicial
subcategory. There is a weak equivalence
\[
N_{\bi}^{+}\pr{\cal C,\cal W}\to\Cls^{+}\pr{N_{\hc}\pr{\cal C},\mor\cal W_{0}}
\]
of $\BS_{\CSS}^{+}$ which is natural in $\pr{\cal C,\cal W}$, where
$\mor\cal W_{0}$ denotes the set of morphisms of $\cal W_{0}$.
\end{thm}

Notice how Theorem \ref{thm:intro} simplifies the clunky marked bisimplicial
set $\Cls^{+}\pr{N_{\hc}\pr{\cal C},\mor\cal W_{0}}$: If we want
to directly work $\Cls^{+}\pr{N_{\hc}\pr{\cal C},\mor\cal W_{0}}$,
we have to construct coherent higher homotopies governing the homotopy
coherent nerve, which is often a hard labor. In contrast, the rows
of $N_{\bi}\pr{\cal C}$ are nerves of \textit{ordinary} categories,
which has no higher structures.
\begin{rem}
\label{rem:DK}The idea of relating the localization of a relative
simplicial category $\pr{\cal C,\cal W}$ (which corresponds to $\Cls^{+}\pr{N_{\hc}\pr{\cal C},\mor\cal W_{0}}$)
to those of $\pr{\cal C_{n},\cal W_{n}}$ (which correspond to the
rows of $N_{\bi}^{+}\pr{\cal C,\cal W}$) is reminiscent of the work
of Dwyer and Kan: In \cite{DK80_1}, Dwyer and Kan defined the \textit{simplicial
localizations} of relative simplicial categories by first defining
them for relative categories, and then applying them levelwise to
define them for all cases. We can therefore interpret Theorem \ref{thm:intro}
as another manifestation of Dwyer and Kan's principle that the localization
of a relative simplicial category is the totality of the levelwise
localization. 
\end{rem}

\begin{rem}
By the works of Joyal \cite{JoyalQCatsCat} and Joyal and Tierney
\cite{JT07}, it has been known that $N_{\bi}\pr{\cal C}$ and $\opn{Cls}\pr{N_{\hc}\pr{\cal C}}$
are weakly equivalent in the complete Segal space model structure.
Theorem \ref{thm:intro} may be regarded as a refinement of this.
\end{rem}

In addition to simplifying the marked classification diagrams, Theorem
\ref{thm:intro} also has some interesting applications. We list two
of them below. 

The first one exploits the connection between localizations and levelwise
localizations we observed in Remark \ref{rem:DK}:
\begin{cor}
[Corollary \ref{cor:appl1}]\label{cor:intro3}Let $\pr{\cal C,\cal W}$
and $\pr{\cal C',\cal W'}$ be relative simplicial categories, where
$\cal C$ and $\cal C'$ are fibrant. Let $f:\cal C\to\cal C'$ be
a simplicial functor which carries $\cal W$ into $\cal W'$. Suppose
that for each $n\geq0$, the functor
\[
N\pr{\cal C_{n}}[N\pr{\cal W_{n}}^{-1}]\to N\pr{\cal C'_{n}}[N\pr{\cal W'_{n}}^{-1}]
\]
is a categorical equivalence (i.e., a weak equivalence in the Joyal model structure). Then so is the functor
\[
N_{\hc}\pr{\cal C}[N_{\hc}\pr{\cal W}^{-1}]\to N_{\hc}\pr{\cal C'}[N_{\hc}\pr{\cal W'}^{-1}].
\]
\end{cor}

Corollary \ref{cor:intro3} is especially useful when $\cal C$ is
equal to the ordinary category $\cal C'_{0}$, the $0$th level of
$\cal C'$. In this case, the corollary says that if the functors
$N\pr{\cal C'_{0}}[N\pr{\cal W_{0}^{\p}}^{-1}]\to N\pr{\cal C'_{n}}[N\pr{\cal W_{n}^{\p}}^{-1}]$
are all categorical equivalences, then so is the functor
\[
N\pr{\cal C'_{0}}N[\pr{\cal W'_{0}}^{-1}]\to N_{\hc}\pr{\cal C'}[N\pr{\cal W'}^{-1}].
\]
In other words, the localization of a relative $\infty$-category
(the right hand side) is equivalent to a localization of an \textit{ordinary}
relative category (the left hand side). This recovers an observation
made by Lurie in the proof of \cite[Proposition 1.3.4.7]{HA}. We
also remark that in the same part of loc.cit., Lurie establishes a
very useful criterion for the maps $N\pr{\cal C'_{0}}[N\pr{\cal W_{0}^{\p}}^{-1}]\to N\pr{\cal C'_{n}}[N\pr{\cal W_{n}^{\p}}^{-1}]$
to be a categorical equivalence: This happens if $\cal C'$ admits
tensoring by $\Delta^{1}$ and $\cal W'$ contains all homotpy equivalences
of $\cal C'$.

The second application relates the homotopy coherent nerve functor
with Segal's classical construction of classifying spaces of simplicial
categories.
\begin{defn}
\label{def:B}\cite{Segal1968} Let $\cal C$ be a simplicial category.
We let $B\pr{\cal C}$ denote the diagonal of the bisimplicial set
$N_{\bi}\pr{\cal C}$; it is the \textbf{classifying space} of $\cal C$.
\end{defn}

In \cite[2.6.1]{Hinich07}, Hinich constructed a natural transformation $B\to
N_{\hc}$. (In fact, there is only one such natural transformation, as we will
see in Proposition \ref{prop:comparison}.) We then prove the following
comparison result:

\begin{cor}
[Corollary \ref{cor:appl2}]\label{cor:intro4}Let $\cal C$ be a
fibrant simplicial category. The map
\[
B\pr{\cal C}\to N_{\hc}\pr{\cal C}
\]
is a weak homotopy equivalence.
\end{cor}

In the special case where $\cal C$ is a simplicial groupoid, Corollary
\ref{cor:intro4} is a consequence of \cite[Theorem 3.6]{A23a} and
\cite[A.5.1]{Hinich01}. (It was also announced in \cite[Corollary
2.6.3]{Hinich07}, but its proof relies on 2.6.2 of loc. cit., which has a gap,
as pointed out in \cite{A23a}.) Our proof of Corollary \ref{cor:intro4} uses
different machinery from these earlier results, and this is why we were able
to relax the hypothesis.

\subsection*{Organization of the paper}

In Section \ref{sec:comparison}, we will construct the comparison
map $B\to N_{\hc}$. In Section \ref{sec:main}, we will prove the
main result. Section \ref{sec:Application} concerns the applications
of the main result.

\subsection*{Notation and convention}
\begin{itemize}
\item If $\cal C$ is a category, its nerve will be denoted by $N\pr{\cal C}$.
\item By a simplicial category, we mean a category enriched over the category
of simplicial sets.
\item If $\cal C$ is a simplicial category and $X,Y\in\cal C$ are its
objects, we will write $\cal C\pr{X,Y}$ or $\Map\pr{X,Y}$ for the
hom-simplicial set from $X$ to $Y$.
\item We understand that the category $\Cat_{\Delta}$ is equipped with
the Bergner model structure \cite{BergnerSimpCat}. 
\item We understand that the category $\SS^{+}$ is equipped with the model
structure for marked simplicial sets \cite[proposition 3.1.3.7, Remark 3.1.4.6]{HA}.
\item If $X$ is a bisimplicial set and $n\geq0$ is an integer, then the
$n$th \textbf{column} (resp. $n$th \textbf{row}) of $X$ is the
simplicial set $X_{n,\ast}$ (resp. $X_{\ast,n}$). The $n$th column
of $X$ is denoted by $X_{n}$.
\item If $\cal C$ is a simplicial category, its \textbf{homotopy coherent
nerve} $N_{\hc}\pr{\cal C}$ is the simplicial set whose $n$-simplices
are the simplicial functors $\fr C[\Delta^{n}]\to\cal C$. Here $\fr C[\Delta^{n}]$
denotes the simplicial category whose objects are the integers $0,\dots,n$.
If $0\leq i\leq j\leq n$ are integers, the hom-simplicial set $\fr C[\Delta^{n}]\pr{i,j}$
is the nerve of the poset $P_{i,j}=\{S\subset\{0,\dots,n\}\mid\min S=i,\,\max S=j\}$,
ordered by inclusion. The composition of $\fr C[\Delta^{n}]$ is induced
by the operation of union.\footnote{It is also common to define $\fr C[\Delta^{n}]\pr{i,j}$ to be the
nerve of the opposite of $P_{i,j}$. Our convention follows \cite{HTT}.
The only part that will be affected by the choice of conventions is
the proof of Proposition \ref{prop:comparison} and the descriptions
of maps appearing in Remarks \ref{rem:intuition} and \ref{rem:adjoint}.}
\end{itemize}

\section{\label{sec:comparison}The comparison map}

In this section, we will construct a comparison map $B\pr{\cal C}\to
N_{\hc}\pr{\cal C}$ (where $B$ is as in Definition \ref{def:B}), which will be
the source of all the other comparison maps we consider in this note. The
comparison map is obtained as the composite of two natural transformation
constructed in \cite[2.6.1]{Hinich07}, but we give a direct approach. In fact,
there is only one natural transformation $B\to N_{\hc}$:
\begin{prop}
\label{prop:comparison}There is a unique natural transformation $B\to N_{\hc}$
of functors $\Cat_{\Delta}\to\SS$.

\end{prop}

For the proof of Proposition \ref{prop:comparison} and for later
discussions, we introduce a bit of notation.
\begin{notation}
Let $n\geq0$ and let $K$ be a simplicial set. We define a simplicial
category $[n]_{K}$ as follows: Its objects are the integers $0,\dots,n$.
The hom-simplicial sets are given by
\[
[n]_{K}\pr{i,j}=\begin{cases}
\prod_{i<k\leq j}K & \text{if }i\leq j,\\
\emptyset & \text{if }i>j.
\end{cases}
\]
The composition is defined by concatenation. 
\end{notation}

\begin{rem}
Let $\cal C$ be a simplicial category and let $K$ be a simplicial
set. We can define an ordinary category $\cal C_{K}$ as follows:
The objects of $\cal C_{K}$ are the objects of $\cal C$. The hom-sets
are given by $\cal C_{K}\pr{X,Y}=\SS\pr{K,\cal C\pr{X,Y}}$. For each
$n\geq0$, a functor $[n]\to\cal C_{K}$ can be identified with a
simplicial functor $[n]_{K}\to\cal C$.
\end{rem}

\begin{proof}
For each $n\geq0$, set $\fr B[\Delta^{n}]=[n]_{\Delta^{n}}$. By
the Yoneda lemma, the simplicial categories $\{\fr B[\Delta^{n}]\}_{n\geq0}$
can be organized into a cosimplicial object of $\Cat_{\Delta}$ in
such a way that the functor $B$ is naturally isomorphic to $\Cat_{\Delta}\pr{\fr B[\Delta^{\bullet}],-}$.
It suffices to show that there is a unique morphism $\fr C[\Delta^{\bullet}]\to\fr B[\Delta^{\bullet}]$
of cosimplicial simplicial categories.

We begin by showing the uniqueness. Suppose there is a map $f:\fr C[\Delta^{\bullet}]\to\fr B[\Delta^{\bullet}]$
of cosimplicial simplicial categories. For each $n\geq0$, the map
$f_{n}:\fr C[\Delta^{n}]\to\fr B[\Delta^{n}]$ must act on the identity
maps on objects because $f_{n}$ is natural in $n$. If $0\leq i\leq j\leq n$
are integers, then the map
\[
f_{n}:\fr C[\Delta^{n}]\pr{i,j}\to\fr B[\Delta^{n}]\pr{i,j}
\]
is completely determined by its values of the vertices, for both $\fr C[\Delta^{n}]\pr{i,j}$
and $\fr B[\Delta^{n}]\pr{i,j}$ are nerves of posets. Since every
vertex of $\fr C[\Delta^{n}]\pr{i,j}$ is a composition of morphisms
in the image of maps $\fr C[\Delta^{1}]\to\fr C[\Delta^{n}]$, we
deduce that $f_{n}$ is completely determined by $f_{1}$. Now there
are exactly two simplicial functors $\fr C[\Delta^{1}]\to\fr B[\Delta^{1}]$
which are bijective on objects (because $\fr C[\Delta^{1}]\pr{0,1}=\Delta^{0}$
and $\fr B[\Delta^{1}]\pr{0,1}=\Delta^{1}$.) If $f_{1}$ carried
the unique vertex of $\fr C[\Delta^{1}]\pr{0,1}$ to the vertex $1\in\fr B[\Delta^{1}]\pr{0,1}$,
then the map
\[
f_{2}:\fr C[\Delta^{2}]\pr{0,2}\to\fr B[\Delta^{2}]\pr{0,2}
\]
would carry the vertices $\{0,2\}$ and $\{0,1,2\}$ to $\pr{2,2}$
and $\pr{2,1}$, respectively. But since there is no edge $\pr{2,2}\to\pr{2,1}$
in $\fr B[\Delta^{2}]\pr{0,2}$, this is impossible. Hence there is
only a unique choice for $f_{1}$, completing the proof of the uniqueness.

For existence, define $f_{n}:\fr C[\Delta^{n}]\to\fr B[\Delta^{n}]$
as follows: On objects, $f_{n}$ acts by the identity map. For each
$0\leq i\leq j\leq n$, the map $\fr C[\Delta^{n}]\pr{i,j}\to\fr B[\Delta^{n}]\pr{i,j}$
is the nerve of the poset map $P_{i,j}\to\underset{j-i\text{ times}}{\underbrace{[n]\times\cdots\times[n]}}$
which assigns to each element $\{i=i_{0}<\cdots<i_{k}=j\}\in P_{i,j}$
the element 
\[
(\underset{i_{k}-i_{k-1}\text{ times}}{\underbrace{i_{k-1},\dots,i_{k-1}}},\dots,\underset{i_{1}-i_{0}\text{ times}}{\underbrace{i_{0},\dots,i_{0}}})\in[n]\times\cdots\times[n].
\]
It is easy to check that the simplicial functors $\{f_{n}\}_{n\geq0}$
indeed define a map of cosimplicial objects of $\Cat_{\Delta}$. The
claim follows.
\end{proof}
\begin{rem}
\label{rem:intuition}Recall that the diagonal of a bisimplicial set
$X$ is equal to the coend $\int^{[n]\in\Del^{\op}}X_{\ast,n}\times\Delta^{n}$.
Therefore, given a simplicial category $\cal C$, the map $B\pr{\cal C}\to N_{\hc}\pr{\cal C}$
of Proposition \ref{prop:comparison} may equally well be specified
by a compatible family of maps $\{\phi_{n}:N\pr{\cal C_{n}}\times\Delta^{n}\to N_{\hc}\pr{\cal C}\}_{n\geq0}$.
Unwinding the definitions, the composite $N\pr{\cal C_{n}}\times\{i\}\hookrightarrow N\pr{\cal C_{n}}\times\Delta^{n}\to N_{\hc}\pr{\cal C}$
is equal to the composite
\[
N\pr{\cal C_{n}}\xrightarrow{i^{*}}N\pr{\cal C_{0}}\to N_{\hc}\pr{\cal C},
\]
where the first map is the restriction along the inclusion $[0]\cong\{i\}\hookrightarrow[n]$
and the second map is induced by the simplicial functor $\cal C_{0}\to\cal C$.
In other words, the map $\phi_{n}$ is the canonical natural transformation
between $n+1$ functors $N\pr{\cal C_{n}}\to N_{\hc}\pr{\cal C}$
corresponding to the $n+1$ elements of $[n]$.
\end{rem}

\begin{rem}
\label{rem:adjoint}Recall that the diagonal of a bisimplicial set
$X$ is equal to the coend $\int^{[n]\in\Del^{\op}}\Delta^{n}\times X_{n,\ast}$.
Therefore, given a simplicial category $\cal C$, the map $B\pr{\cal C}\to N_{\hc}\pr{\cal C}$
of Proposition \ref{prop:comparison} may equally well be specified
by a compatible family of maps $\{\psi_{n}:\Delta^{n}\times N_{\bi}\pr{\cal C}_{n}\to N_{\hc}\pr{\cal C}\}_{n\geq0}$.
Unwinding the definitions, the map $\psi_{n}$ admits the following
description:
\begin{enumerate}
\item If $n=0$, then $\psi_{n}$ is the inclusion $N_{\bi}\pr{\cal C}_{0}\cong\ob\cal C\hookrightarrow N_{\hc}\pr{\cal C}$.
\item Let $n\geq1$ and let $\sigma:\Delta^{m}\to N_{\bi}\pr{\cal C}_{n}$
be an $m$-simplex of $N_{\bi}\pr{\cal C}_{n}$, corresponding to
a simplicial functor $\sigma':[n]_{\Delta^{m}}\to\cal C$. Then the
composite
\[
\Delta^{n}\times\Delta^{m}\xrightarrow{\id\times\sigma}\Delta^{n}\times N_{\bi}\pr{\cal C}_{n}\xrightarrow{\psi_{n}}N_{\hc}\pr{\cal C}
\]
is adjoint to the composite
\[
\fr C[\Delta^{n}\times\Delta^{m}]\xrightarrow{\chi}[n]_{\Delta^{m}}\xrightarrow{\sigma'}\cal C.
\]
Here $\chi$ is defined on objects by $\chi\pr{i,j}=i$. Given integers
$0\leq i\leq i'\leq n$ and $0\leq j\leq j'\leq m$, let $P_{\pr{i,j},\pr{i',j'}}$
denote the poset of linearly ordered subsets $S\subset[n]\times[m]$
such that $\min S=\pr{i,j}$ and $\max S=\pr{i',j'}$. Then the simplicial
set $\fr C[\Delta^{n}\times\Delta^{m}]\pr{\pr{i,j},\pr{i',j'}}$ is
the nerve of $P_{\pr{i,j},\pr{i',j'}}$, and the map $\fr C[\Delta^{n}\times\Delta^{m}]\pr{\pr{i,j},\pr{i',j'}}\to[n]_{\Delta^{m}}\pr{i,i'}$
is induced by the poset map
\begin{align*}
P_{\pr{i,j},\pr{i',j'}} & \to\prod_{i<p\leq i'}[m],\\
S & \mapsto\pr{\max\{q\in[m]\mid\pr{i'',q}\in S\text{ for some }i''<p\}}_{i<p\leq i'}.
\end{align*}
\end{enumerate}
\end{rem}

\section{\label{sec:main}Main result}

In this section, we state and prove the main result of this paper.

Here is the statement of the main result.
\begin{thm}
\label{thm:main}Let $\cal C$ be a fibrant simplicial category and
let $\cal W\subset\cal C$ be a wide simplicial subcategory. The maps
$\{\psi_{n}:\Delta^{n}\times N_{\bi}\pr{\cal C}_{n}\to N_{\hc}\pr{\cal C}\}_{n\geq0}$
of Remark \ref{rem:adjoint} induces a weak equivalence
\[
\theta:N_{\bi}^{+}\pr{\cal C,\cal W}\to\Cls^{+}\pr{N_{\hc}\pr{\cal C},\mor\cal W_{0}}
\]
of $\BS_{\CSS}^{+}$.
\end{thm}

The remainder of this section is devoted to the proof of Theorem \ref{thm:main}.
We begin by establishing the unmarked version of Theorem \ref{thm:main}:
\begin{prop}
\label{prop:main_um}Let $\cal C$ be a fibrant simplicial category.
The maps $\{\psi_{n}:\Delta^{n}\times N_{\bi}\pr{\cal C}_{n}\to N_{\hc}\pr{\cal C}\}_{n\geq0}$
of Remark \ref{rem:adjoint} induces a weak equivalence
\[
N_{\bi}\pr{\cal C}\to\Cls\pr{N_{\hc}\pr{\cal C}}
\]
of $\BS_{\CSS}$.
\end{prop}

\begin{proof}
Observe that the Reedy fibrant replacement $N_{\bi}^{f}\pr{\cal C}$
of $N_{\bi}\pr{\cal C}$ is a Segal space; indeed, for each $n\geq2$,
the square 
\[\begin{tikzcd}
    {N_{\mathrm{bi}}(\mathcal{C})_n} & {N_{\mathrm{bi}}(\mathcal{C})_{n-1}} \\
    {N_{\mathrm{bi}}(\mathcal{C})_1} & {N_{\mathrm{bi}}(\mathcal{C})_0}
    \arrow[from=1-1, to=1-2]
    \arrow[from=1-1, to=2-1]
    \arrow[from=1-2, to=2-2]
    \arrow[from=2-1, to=2-2]
\end{tikzcd}\]induced by the inclusions $[1]\cong\{n-1<n\}\hookrightarrow[n]\hookleftarrow[n-1]$
is homotopy cartesian. Therefore, by \cite[Theorem 7.7]{Rezk01},
it suffices to show that the induced map $N_{\bi}^{f}\pr{\cal C}\to\opn{Cls}\pr{N_{\hc}\pr{\cal C}}$
is a Dwyer--Kan equivalence. Since the map $N_{\bi}\pr{\cal C}_{0,0}\to\Cls\pr{N_{\hc}\pr{\cal C}}_{0,0}$
is bijective, it suffices to show that the square 
\[\begin{tikzcd}
    {N_{\mathrm{bi}}(\mathcal{C})_1} & {\operatorname{Cls}(N_{\mathrm{hc}}(\mathcal{C}))_1} \\
    {N_{\mathrm{bi}}(\mathcal{C})_0\times N_{\mathrm{bi}}(\mathcal{C})_0} & {\operatorname{Cls}(N_{\mathrm{hc}}(\mathcal{C}))_0\times \operatorname{Cls}(N_{\mathrm{hc}}(\mathcal{C}))_0}
    \arrow[from=1-1, to=1-2]
    \arrow[from=1-2, to=2-2]
    \arrow[from=1-1, to=2-1]
    \arrow[from=2-1, to=2-2]
\end{tikzcd}\]is homotopy cartesian. For each pair of objects $X,Y\in\cal C$, the
induced map between the fibers of the vertical arrows over $\pr{X,Y}$
can be identified with the homotopy equivalence
\[
\cal C\pr{X,Y}\to\Hom_{N_{\hc}\pr{\cal C}}\pr{X,Y}
\]
of \cite[\href{https://kerodon.net/tag/01LF}{Tag 01LF}]{kerodon}
(see Remark \ref{rem:adjoint}). Hence the square is homotopy cartesian,
as required.
\end{proof}
Before we proceed, we recall a few facts on marked bisimplicial sets.
\begin{notation}
\cite[Definition 3.1]{A23b}\label{nota:t+} We let $\diag^{+}:\BS^{+}\to\SS^{+}$
denote the functor
\[
\diag^{+}\pr{X,S}=\pr{\diag\pr X,S_{1}},
\]
where $\diag\pr X=\{X_{n,n}\}_{n\geq0}$ denotes the diagonal of $X$.
The functor $\diag^{+}$ is the left adjoint of the functor $\Cls^{+}:\SS^{+}\to\BS^{+}$.
\end{notation}

\begin{thm}
\label{thm:CSS+}There is a simplicial model structure on $\BS^{+}$,
denoted by $\BS_{\CSS}^{+}$, which has the following properties:
\begin{enumerate}
\item The cofibrations are the monomorphisms.
\item The fibrant objects are the marked bisimplicial sets of the form $\cal X^{\natural}=\pr{\cal X,\cal X_{{\rm hoeq}}}$,
where $\cal X$ is a complete Segal space and $\cal X_{{\rm hoeq}}\subset\cal X_{1}$
is the union of components spanned by homotopy equivalences of $\cal X$.
\end{enumerate}
Moreover, the model structure has the following additional properties:

\begin{enumerate}
\setcounter{enumi}{2}

\item If $\cal X$ is a complete Segal space and $\pr{A,S}$ is a
marked bisimplicial set, then $\Map\pr{\pr{A,S},\cal X^{\natural}}$
is the component of $\Map\pr{A,\cal X}$ spanned by the maps $A\to\cal X$
which carries $S$ into $\cal X_{{\rm hoeq}}$. Here $\Map\pr{A,\cal X}$
denotes the simplicial enrichment of $\BS$ adapted to the complete
Segal space model structure \cite[$\S$2.3]{Rezk01}.

\item The adjunction $\diag^{+}:\BS_{\CSS}^{+}\adj\SS^{+}:\Cls^{+}$
is a Quillen equivalence.

\item The model structure $\BS_{\CSS}^{+}$ is a Bousfield localization
of the Reedy model structure on $\BS^{+}=\pr{\SS^{+}}^{\Del^{\op}}$.

\item The functor $\Cls^{+}:\SS^{+}\to\BS_{\CSS}^{+}$ preserves
and reflects all weak equivalences. 

\end{enumerate}
\end{thm}

\begin{proof}
The first half is established in \cite[Theorem 2.9]{A23b}. Point
(3) is proved in \cite[Remark 2.7]{A23b}, points (4) and (5) are
proved in \cite[Theorem 3.4]{A23b}, and point (6) is proved in \cite[Theorem 4.2]{A23b}.
\end{proof}
\begin{notation}
Let $\pr{\cal C,\cal W}$ be a relative simplicial category. We define
a marked bisimplicial set $B^{+}\pr{\cal C,\cal W}$ by 
\[
B^{+}\pr{\cal C,\cal W}=\diag^{+}N_{\bi}^{+}\pr{\cal C,\cal W}.
\]
\end{notation}

We now arrive at the proof of Theorem \ref{thm:main}.
\begin{proof}
[Proof of Theorem \ref{thm:main}]We prove the theorem in four steps.

\begin{enumerate}[label=(Step \arabic*), leftmargin = 0em, itemindent = 4em]

\item Assume first that $\cal W\subset\cal C$ is the smallest wide
simplicial subcategory containing all homotopy equivalences of $\cal C$.
If $\cal X$ is a complete Segal space, then every morphism $N_{\bi}\pr{\cal C}\to\cal X$
of bisimplicial sets induces a map $N_{\bi}^{+}\pr{\cal C,\cal W}\to\cal X^{\natural}$
of marked bisimplicial sets. Indeed, we only have to show that the
induced map $N\pr{\cal C_{0}}\to\cal X_{\ast,0}$ between the $0$th
row respects the markings (because $\cal X_{{\rm hoeq}}\subset\cal X$
is a union of components), which is obvious. Likewise, any map $\opn{Cls}\pr{N_{\hc}\pr{\cal C}}\to\cal X$
lifts to a morphism $\Cls^{+}\pr{N_{\hc}\pr{\cal C},N_{\hc}\pr{\cal W}}\to\cal X^{\natural}$.
Therefore, by properties (1) through (3) of Theorem \ref{thm:CSS+},
it suffices to show that the underlying map 
\[
N_{\bi}\pr{\cal C}\to\opn{Cls}\pr{N_{\hc}\pr{\cal C}}
\]
of $\theta$ is a weak equivalence of $\BS_{\CSS}$. This is nothing
but Proposition \ref{prop:main_um}.

\item According to Theorem \ref{thm:CSS+}, the functor $\Cls^{+}:\SS^{+}\to\BS_{\CSS}^{+}$
is a right Quillen equivalence and preserves all weak equivalences.
Therefore, $\theta$ is a weak equivalence of $\BS_{\CSS}^{+}$ if
and only if the map
\[
B^{+}\pr{\cal C,\cal W}\to\pr{N_{\hc}\pr{\cal C},\mor\cal W_{0}}
\]
which is adjoint to $\theta$ is a weak equivalence of $\SS^{+}$.

\item Suppose that $\cal W$ contains all homotopy equivalences of
$\cal C$. We observe that:
\begin{itemize}
\item The marked edges of $B^{+}\pr{\cal C,\cal W}$ are precisely the inverse
image of $\mor\cal W_{0}$ under the map $B\pr{\cal C}\to N_{\hc}\pr{\cal C}$;
this is because $\pr{\cal C,\cal W}$ is a relative simplicial category.
\item The map $B\pr{\cal C}\to N_{\hc}\pr{\cal C}$ induces a surjection
between the set of edges; this follows by inspection.
\end{itemize}
Since $\mor\cal W_{0}$ contains all equivalences of $N_{\hc}\pr{\cal C}$,
the claim now follows by combining the above observations, Steps 1
and 2, and the definition of weak equivalences of marked simplicial
sets.

\item We prove the theorem in the general case. Let $\cal W'\subset\cal C$
denote the smallest wide simplicial subcategory containing $\cal W$
and all homotopy equivalences of $\cal C$. The map $N_{\bi}^{+}\pr{\cal C,\cal W}\to N_{\bi}^{+}\pr{\cal C,\cal W'}$
is a weak equivalence of $\BS_{\CSS}^{+}$, and the map $\pr{N_{\hc}\pr{\cal C},\mor\cal W_{0}}\to\pr{N_{\hc}\pr{\cal C},\mor\cal W_{0}'}$
is a weak equivalence of $\SS^{+}$. The claim now follows from part
(6) of Theorem \ref{thm:CSS+}.

\end{enumerate}
\end{proof}
\begin{rem}
Let $\cal C$ be a fibrant simplicial category and let $\cal W\subset\cal C$
be a wide simplicial subcategory. As we saw in Step 2 of the proof
of Theorem \ref{thm:main}, Theorem \ref{thm:main} implies that the
map
\[
B^{+}\pr{\cal C,\cal W}\to\pr{N_{\hc}\pr{\cal C},\mor\cal W_{0}}
\]
is a weak equivalence of marked simplicial sets. Since the geometric
realization functor models homotopy colimits in simplicial model categories
(\cite[Theorem 5.2.3]{cathtpy}, \cite[Lemma 15.3.9]{Hirschhorn}),
and since $\diag^{+}:\BS^{+}\to\SS^{+}$ is nothing but the geometric
realization functor, we may interpret Theorem \ref{thm:main} as saying
that the localization $N_{\hc}\pr{\cal C}[N_{\hc}\pr{\cal W}^{-1}]$
is a homotopy colimit of the simplicial $\infty$-category $[n]\mapsto N\pr{\cal C_{n}}[N\pr{\cal W_{n}}^{-1}]$.
This point of view was articulated by Lurie in \cite[Proposition 1.3.4.14]{HA};
in fact, Theorem \ref{thm:main} can also be proved using Lurie's
result (and Theorem \ref{thm:CSS+}), though the proof will be a little
longer.
\end{rem}

\section{\label{sec:Application}Applications}

We now list two applications of Theorem \ref{thm:main}.
\begin{cor}
\label{cor:appl1}Let $\pr{\cal C,\cal W}$ and $\pr{\cal C',\cal W'}$
be relative simplicial categories, where $\cal C$ and $\cal C'$
are fibrant. Let $f:\cal C\to\cal C'$ be a simplicial functor which
carries $\cal W$ into $\cal W'$. If for each $n\geq0$, the map
\[
\pr{N\pr{\cal C_{n}},\mor\cal W_{n}}\to\pr{N\pr{\cal C'_{n}},\mor\cal W'_{n}}
\]
is a weak equivalence of marked simplicial sets, then the map
\[
\pr{N_{\hc}\pr{\cal C},\mor\cal W_{0}}\to\pr{N_{\hc}\pr{\cal C'},\mor\cal W'_{0}}
\]
is also a weak equivalence.
\end{cor}

\begin{proof}
By hypothesis, the map $N_{\bi}^{+}\pr{\cal C,\cal W}\to N_{\bi}^{+}\pr{\cal C',\cal W'}$
induces a weak equivalence of marked simplicial sets in each row.
Thus, by part (5) of Theorem \ref{thm:CSS+}, it is a weak equivalence
of $\BS_{\CSS}^{+}$. It follows from Theorem \ref{thm:main} that
the map $\Cls^{+}\pr{N_{\hc}\pr{\cal C},\mor\cal W_{0}}\to\Cls^{+}\pr{N_{\hc}\pr{\cal C'},\mor\cal W'_{0}}$
is a weak equivalence of $\BS_{\CSS}^{+}$. Using part (6) of Theorem
\ref{thm:CSS+}, we deduce that the map $\pr{N_{\hc}\pr{\cal C},\mor\cal W_{0}}\to\pr{N_{\hc}\pr{\cal C'},\mor\cal W'_{0}}$
is also a weak equivalence of $\SS^{+}$, and we are done.
\end{proof}

\begin{rem} 
We do not know if the converse of Corollary \ref{cor:appl1} holds. We expect
that this is false, given that the proof of the corollary relies on point (5)
of Theorem \ref{thm:CSS+}, which only gives a \textit{sufficient} condition
for a map of $\BS_{\CSS}^{+}$ to be a weak equivalence. However, we are not
aware of explicit counterexamples.
\end{rem}

\begin{cor}
\label{cor:appl2}Let $\cal C$ be a fibrant simplicial category.
The map
\[
B\pr{\cal C}\to N_{\hc}\pr{\cal C}
\]
of Proposition \ref{prop:comparison} is a weak homotopy equivalence.
\end{cor}

\begin{proof}
By definition, every weak equivalence of marked simplicial sets induces
a weak homotopy equivalence between the underlying simplicial sets
\cite[Proposition 3.1.3.3]{HTT}. It will therefore suffice to show
that the map $B^{+}\pr{\cal C,\cal C^{\simeq}}\to\pr{N_{\hc}\pr{\cal C},\pr{\cal C^{\simeq}}_{0}}=N_{\hc}\pr{\cal C}^{\natural}$
is a weak equivalence of marked simplicial sets, where $\cal C^{\simeq}\subset\cal C$
denotes the smallest wide simplicial subcategory containing all equivalences
of $\cal C$. This is immediate from Theorem \ref{thm:main} (and
part (4) of Theorem \ref{thm:CSS+}).
\end{proof}

\section*{Acknowledgment}

The author is grateful to the anonymous referee for carefully reading the
manuscript and making valuable suggestions. He also appreciates Daisuke
Kishimoto and Mitsunobu Tsutaya for commenting on earlier drafts of this
paper.


\begin{thebibliography}{10}

\bibitem{A23b}
K.~Arakawa.
\newblock Classification diagrams of marked simplicial sets.
\newblock \url{https://arxiv.org/abs/2311.01101}, 2023.

\bibitem{A23a}
K.~Arakawa.
\newblock Classifying space via homotopy coherent nerve.
\newblock {\em Homology Homotopy Appl.}, 25(2):373--381, 2023.

\bibitem{AC25}
K.~Arakawa and B.~Cnossen.
\newblock A short proof of the universality of the relative Rezk nerve.
\newblock \url{https://arxiv.org/abs/2505.14123}, 2025.
\newblock To appear in {\em Proc. Amer. Math. Soc.}

\bibitem{MR2877402}
C.~Barwick and D.~M. Kan.
\newblock A characterization of simplicial localization functors and a
  discussion of {DK} equivalences.
\newblock {\em Indag. Math. (N.S.)}, 23(1-2):69--79, 2012.

\bibitem{BergnerSimpCat}
J.~E. Bergner.
\newblock A model category structure on the category of simplicial categories.
\newblock {\em Trans. Amer. Math. Soc.}, 359(5):2043--2058, 2007.

\bibitem{MR2439415}
J.~E. Bergner.
\newblock Complete {S}egal spaces arising from simplicial categories.
\newblock {\em Trans. Amer. Math. Soc.}, 361(1):525--546, 2009.

\bibitem{Cordier_hcnerve}
J.-M. Cordier.
\newblock Sur la notion de diagramme homotopiquement coh\'erent.
\newblock {\em Cahiers de Topologie et G\'eom\'etrie Diff\'erentielle
  Cat\'egoriques}, 23(1):93--112, 1982.

\bibitem{DK80_1}
W.~G. Dwyer and D.~M. Kan.
\newblock Simplicial localizations of categories.
\newblock {\em J. Pure Appl. Algebra}, 17(3):267--284, 1980.

\bibitem{Hinich01}
V.~Hinich.
\newblock D{G} coalgebras as formal stacks.
\newblock {\em J. Pure Appl. Algebra}, 162(2-3):209--250, 2001.

\bibitem{Hinich07}
V.~Hinich.
\newblock Homotopy coherent nerve in deformation theory.
\newblock \url{https://arxiv.org/abs/0704.2503}, 2007.

\bibitem{Hirschhorn}
P.~Hirschhorn.
\newblock {\em Model Categories and Their Localizations}.
\newblock Mathematical Surveys and Monographs. American Mathematical Society,
  2003.

\bibitem{Joyal_qcat_Kan}
A.~Joyal.
\newblock Quasi-categories and {K}an complexes.
\newblock {\em J. Pure Appl. Algebra}, 175(1-3):207--222, 2002.
\newblock Special volume celebrating the 70th birthday of Professor Max Kelly.

\bibitem{JoyalQCatsCat}
A.~Joyal.
\newblock Quasi-categories vs simplicial categories.
\newblock
  \url{https://www.math.uchicago.edu/~may/IMA/Incoming/Joyal/QvsDJan9(2007).pdf},
  2007.

\bibitem{JT07}
A.~Joyal and M.~Tierney.
\newblock Quasi-categories vs {S}egal spaces.
\newblock In {\em Categories in algebra, geometry and mathematical physics},
  volume 431 of {\em Contemp. Math.}, pages 277--326. Amer. Math. Soc.,
  Providence, RI, 2007.

\bibitem{Landoo-cat}
M.~Land.
\newblock {\em Introduction to infinity-categories}.
\newblock Compact Textbooks in Mathematics. Birkh\"auser/Springer, Cham, [2021]
  \copyright2021.

\bibitem{LM15}
Z.~L. Low and A.~Mazel-Gee.
\newblock From fractions to complete {S}egal spaces.
\newblock {\em Homology Homotopy Appl.}, 17(1):321--338, 2015.

\bibitem{HTT}
J.~Lurie.
\newblock {\em Higher topos theory}, volume 170 of {\em Annals of Mathematics
  Studies}.
\newblock Princeton University Press, Princeton, NJ, 2009.

\bibitem{HA}
J.~Lurie.
\newblock Higher algebra.
\newblock \url{https://www.math.ias.edu/~lurie/papers/HA.pdf}, 2017.

\bibitem{kerodon}
J.~Lurie.
\newblock Kerodon.
\newblock \url{https://kerodon.net}, 2024.

\bibitem{MR4045352}
A.~Mazel-Gee.
\newblock The universality of the {R}ezk nerve.
\newblock {\em Algebr. Geom. Topol.}, 19(7):3217--3260, 2019.

\bibitem{Rezk01}
C.~Rezk.
\newblock A model for the homotopy theory of homotopy theory.
\newblock {\em Trans. Amer. Math. Soc.}, 353(3):973--1007, 2001.

\bibitem{cathtpy}
E.~Riehl.
\newblock {\em Categorical homotopy theory}, volume~24 of {\em New Mathematical
  Monographs}.
\newblock Cambridge University Press, Cambridge, 2014.

\bibitem{Segal1968}
G.~Segal.
\newblock Classifying spaces and spectral sequences.
\newblock {\em Inst. Hautes \'Etudes Sci. Publ. Math.}, (34):105--112, 1968.

\end{thebibliography}
\end{document}